\documentclass[12pt]{article}

\usepackage[margin=1.5in]{geometry}

\usepackage{amsmath}
\usepackage{amsthm}
\usepackage{amsfonts}
\usepackage{eucal}
\usepackage{multicol}
\usepackage[numbers,sort&compress]{natbib}
\usepackage{booktabs}
\usepackage{titlesec}

\usepackage{enumitem}
\setlist[enumerate]{label={\upshape(\alph*)}}

\usepackage{color}

\newcommand{\den}{\mbox{Den}}

\usepackage{tikz}
\usetikzlibrary{calc}
\tikzstyle{vertex}=[circle, draw, inner sep=0pt, minimum size=4pt,fill=black]
\newcommand{\vertex}{\node[vertex]}
\tikzstyle{hollowvertex}=[circle, draw, inner sep=0pt, minimum size=4pt]

\tikzstyle{namedvertex}=[circle, draw, inner sep=1pt, minimum size=12pt]

\tikzstyle{phantomvertex}=[circle, draw, inner sep=0pt, minimum size=4pt,color=white]

\usetikzlibrary{decorations.pathreplacing}
\usetikzlibrary{backgrounds}
\usetikzlibrary{arrows.meta}
\tikzset{
  .../.tip={[sep=2pt 2]
    Round Cap[]. Circle[length=0pt 2,sep=2pt] Circle[length=0pt 2,sep=2pt] Circle[length=0pt 2, sep=2pt 2]}}

\newtheorem{theorem}{Theorem}[section]
\newtheorem{lemma}[theorem]{Lemma}
\newtheorem{proposition}[theorem]{Proposition}

\theoremstyle{definition}

\newtheorem{remark}[theorem]{Remark}
\newtheorem{conjecture}[theorem]{Conjecture}
\newtheorem{problem}{Problem}

\begin{document}

\title{On the mean subtree order of graphs under edge addition}
\author{Ben Cameron\\
\small School of Computer Science\\
\small University of Guelph\\
\small ben.cameron@uoguelph.ca\\
\\
Lucas Mol\\
\small Department of Mathematics and Statistics\\
\small The University of Winnipeg\\
\small l.mol@uwinnipeg.ca}
\date{November 2019}

\maketitle

\begin{abstract}
\noindent
For a graph $G$, the \textit{mean subtree order} of $G$ is the average order of a subtree of $G$. In this note, we provide counterexamples to a recent conjecture of Chin, Gordon, MacPhee, and Vincent, that for every connected graph $G$ and every pair of distinct vertices $u$ and $v$ of $G$, the addition of the edge between $u$ and $v$ increases the mean subtree order. In fact, we show that the addition of a single edge between a pair of nonadjacent vertices in a graph of order $n$ can decrease the mean subtree order by as much as $n/3$ asymptotically. We propose the weaker conjecture that for every connected graph $G$ which is not complete, there exists a pair of nonadjacent vertices $u$ and $v$, such that the addition of the edge between $u$ and $v$ increases the mean subtree order.  We prove this conjecture in the special case that $G$ is a tree.

\end{abstract}

\section{Introduction}

Throughout, we assume that \emph{graphs} are finite, loopless, and contain no multiple edges, while \emph{multigraphs} are finite, loopless, and may contain multiple edges.  In particular, every graph is a multigraph.  A \emph{subtree} of a multigraph $G$ is a (not necessarily induced) subgraph of $G$ which is a tree.  By convention, the null graph is not considered a subtree of any multigraph.  The study of subtrees of trees goes back at least to Jamison~\cite{Jamison1983, Jamison1984}, whose work on the mean order of the subtrees of a tree has received considerable attention in the last decade~\cite{VinceWang2010,WagnerWang2014,WagnerWang2016,Haslegrave2014,MolOellermann2019}. 

Recently, Chin, Gordon, MacPhee, and Vincent~\cite{ChinGordonMacpheeVincent2018} initiated the study of subtrees of multigraphs in general. For a given multigraph $G$, two parameters introduced by Chin et al.~are the mean subtree order of $G$, denoted $\mu(G)$, and the proportion of subtrees of $G$ that are spanning, denoted $P(G)$.  Among other things, Chin et al.~proved that $P(K_n)$ tends to $e^{-1/e}$, and that $P(K_{n,n})$ tends to $e^{-2/e}$, as $n$ tends to $\infty$.  They also suggested many interesting problems and conjectures.

Several of these conjectures on the proportion of subtrees that are spanning have very recently been resolved by Wagner~\cite{Wagner2019}.  Wagner has shown that if $G_n$ is a sequence of sparse random graphs (i.e., if $G_n$ is the Erd\H{o}s-Renyi graph $G(n,p_n)$, where $p_n\rightarrow 0$), then $P(G_n)$ tends to $0$.  On the other hand, if $G_n$ is a sequence of dense random graphs (i.e., if $G_n=G(n,p_n)$, where $p_n\rightarrow p_\infty >0$), then $P(G_n)$ tends to the positive number $e^{-1/ep_{\infty}}$.

In this note, we are concerned with the following conjecture of Chin et al.~\cite{ChinGordonMacpheeVincent2018}, and some related problems.

\begin{conjecture}[Conjecture 7.4 in \cite{ChinGordonMacpheeVincent2018}]\label{ChinConjecture}
Suppose that $G$ is a connected multigraph, and that $H$ is obtained from $G$ by adding an edge
between two distinct vertices of $G$. Then $\mu(G)<\mu(H)$.
\end{conjecture}

\noindent
We will be most interested in the case that $G$ and $H$ are graphs, i.e., that $G$ contains no multiple edges, and $H$ is obtained from $G$ by adding an edge between two distinct, nonadjacent vertices.  As pointed out by Chin et al., if Conjecture~\ref{ChinConjecture} were true, it would follow immediately that the complete graph has the largest mean subtree order among all connected graphs of a given order -- this problem is still open.

However, some small counterexamples to Conjecture~\ref{ChinConjecture} (where $G$ and $H$ are both graphs) arise from a computer search.  The graph $G$ of order $7$ shown in Figure~\ref{CounterexampleFigure} is the unique connected graph up to isomorphism of order at most $7$ for which Conjecture~\ref{ChinConjecture} fails.  The graph obtained from $G$ by adding an edge between the vertices $a$ and $b$ has mean subtree order approximately $0.000588$ smaller than the mean subtree order of $G$.  Up to isomorphism, there are $347$ graphs of order $8$ for which Conjecture~\ref{ChinConjecture} fails, though the largest decrease in the mean subtree order among all of these counterexamples is still small at approximately $0.0395$. We can even exhibit an infinite family of counterexamples to Conjecture~\ref{ChinConjecture} fairly easily.  If $n\geq 8$, and $H_n$ is the graph obtained from $K_{2,n-2}$ by joining the vertices in the partite set of cardinality $2$, then it can be shown that $\mu(H_n)<\mu(K_{2,n-2})$. However, the difference $\mu(K_{2,n-2})-\mu(H_n)$ tends to $0$ as $n$ tends to $\infty$, i.e., the decrease in the mean subtree order becomes arbitrarily small.

This raises the question: If $H$ is obtained from a graph $G$ by adding an edge between a pair of distinct, nonadjacent vertices, then how large can the difference $\mu(G)-\mu(H)$ be? We show that this difference can grow linearly in the order of $G$. More precisely, we show that  if $G$ has order $n$, then $\mu(G)-\mu(H)$ can be as large as $n/3$ asymptotically.

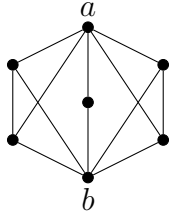
\begin{figure}
    \centering
    \begin{tikzpicture}
    \vertex (a) at (0,2) {};
    \node[above] at (a) {$a$};
    \vertex (b) at (0,0) {};
    \node[below] at (b) {$b$};
    \vertex (c) at (0,1) {};
    \vertex (d) at (-1,0.5) {};
    \vertex (e) at (-1,1.5) {};
    \vertex (f) at (1,0.5) {};
    \vertex (g) at (1,1.5) {};
    \path
    (a) edge (c)
    (b) edge (c)
    (a) edge (d)
    (a) edge (e)
    (a) edge (f)
    (a) edge (g)
    (b) edge (d)
    (b) edge (e)
    (b) edge (f)
    (b) edge (g)
    (d) edge (e)
    (f) edge (g);
    \end{tikzpicture}
    \caption{The smallest counterexample to Conjecture~\ref{ChinConjecture}.  Adding an edge between $a$ and $b$ decreases the mean subtree order.}
    \label{CounterexampleFigure}
\end{figure}

Although Conjecture~\ref{ChinConjecture} is false, we propose the following weaker conjecture, which we prove in the special case that $G$ is a tree.

\begin{conjecture}\label{AddConjecture}
Suppose $G$ is a connected graph which is not complete. Then there is a graph $H$, obtained from $G$ by joining two distinct, nonadjacent vertices, such that $\mu(H)>\mu(G).$
\end{conjecture}

\noindent
The truth of Conjecture~\ref{AddConjecture} would still imply that the complete graph has the largest mean subtree order among all connected graphs of a given order.

The layout of the remainder of the note is as follows.  In Section~\ref{PrelimSection}, we provide the necessary background and notation.  In Section~\ref{CounterexampleSection}, we demonstrate that the addition of a single edge to a graph of order $n$ can decrease the mean subtree order by as much as $n/3$ asymptotically.  In Section~\ref{TreeSection}, we prove Conjecture~\ref{AddConjecture} in the special case that $G$ is a tree.  We also show that Conjecture~\ref{AddConjecture} holds for every graph $G$ if we drop the condition that the two vertices being joined are nonadjacent (i.e., we can increase the mean subtree order of $G$ by adding a new edge parallel to an existing edge in $G$).

\section{Preliminaries}\label{PrelimSection}

Let $G$ be a graph of order $n$, and let $\mathcal{T}_G$ be the set of subtrees of $G$.  The \emph{subtree polynomial} of $G$, denoted $S_G(x)$, is given by
\[
S_G(x)=\sum_{T\in\mathcal{T}_G}x^{|V(T)|}.
\]
The \emph{mean subtree order} of $G$, denoted $\mu(G)$, is the average order of a subtree of $G$.  It is straightforward to show that
\[
\mu(G)=\frac{S'_G(1)}{S_G(1)}.
\]
The \emph{density} of $G$, denoted $\den(G)$, is the mean subtree order of $G$ divided by the order of $G$, that is,
\[
\den(G)=\frac{\mu(G)}{n}.
\]

\begin{remark}
Chin et al.~\cite{ChinGordonMacpheeVincent2018} focused on the \emph{size} of the subtrees of a graph, while we focus on their \emph{order}, which is more in line with the work of Jamison~\cite{Jamison1983}.  It is straightforward to switch back and forth between these two points of view.  In particular, the mean subtree order of $G$ (denoted $\mu(G)$ here) is equal to one plus the mean subtree size of $G$ (denoted $\mu(G)$ by Chin et al.~\cite{ChinGordonMacpheeVincent2018}).  
\end{remark}

Let $G$ be a graph of order $n$ and let $p$ be either a vertex or an edge of $G$.  Let $\mathcal{T}_{G,p}$ be the set of subtrees of $G$ containing $p$.  The \emph{local subtree polynomial} of $G$ at $p$, denoted $S_{G,p}(x)$, is given by
\[
S_{G,p}(x)=\sum_{T\in \mathcal{T}_{G,p}}x^{|V(T)|}.
\]
The \emph{local mean subtree order} of $G$ at $p$, denoted $\mu(G,p)$, and
the \emph{local density} of $G$ at $p$, denoted $\den(G,p)$, are defined analogously to the global versions given above:
\[
\mu(G,p)=\frac{S'_{G,p}(1)}{S_{G,p}(1)}, \mbox{ and } \den(G,p)=\frac{\mu(G,p)}{n}.
\]

The \textit{logarithmic derivative} of a function $f(x)$ is defined by $\frac{f'(x)}{f(x)}$ for all values of $x$ for which $f$ is differentiable and nonzero. For any function $f(x)$ whose logarithmic derivative exists at $1$, we abuse notation slightly and let $\mu(f(x))$ denote the logarithmic derivative of $f(x)$ evaluated at $1$, i.e., $\mu(f(x))=f'(1)/f(1).$  So, for example, we have $\mu(S_G(x))=\mu(G)$.  This notation is particularly useful when we can write a (local) subtree polynomial as a product of other polynomials, because of the following elementary lemma.

\begin{lemma}\label{logsum}
Let $f$, $g$, and $h$ be functions whose logarithmic derivatives exist at $1$.  If $f(x)=g(x)h(x)$, then
\[
\mu(f(x))=\mu(g(x))+\mu(h(x)).
\]
\end{lemma}

We will require the following straightforward lemma.

\begin{lemma}\label{LocalCycleLemma}
Let $n\geq 3$ and let $e$ be an edge of $C_n$.  Then
\[
S_{C_n,e}(x)=x^2\sum_{i=0}^{n-2}(i+1)x^i.
\]
In particular, we have
\begin{align*}
S_{C_n,e}(1)=\binom{n}{2}, \ \ \ S_{C_n,e}'(1)=\frac{n(n-1)(2n+2)}{6} \ \ \ \mbox{and} \ \ \ \mu(C_n,e)=\frac{2n+2}{3}.
\end{align*}
\end{lemma}

\begin{proof}
First of all, the spanning subtrees of $C_n$ that contain $e$ correspond to the $n-1$ edges of $C_n$ not equal to $e$.  Now let $u$ and $v$ be the vertices incident to $e$ in $C_n$, and let $P_{n-2}$ be the path obtained by deleting $u$ and $v$ from $C_n$.  The non-spanning subtrees of $C_n$ containing $e$ correspond to the subtrees of $P_{n-2}$.  For $0\leq i\leq n-3$, there are exactly $i+1$ subtrees of $P_{n-2}$ of order $n-i-2$, and hence there are exactly $i+1$ subtrees of $C_n$ of order $i+2$ that contain $e$.  Thus we have
\begin{align*}
S_{C_n,e}(x)&=(n-1)x^n+\sum_{i=0}^{n-3}(i+1)x^{i+2}=\sum_{i=0}^{n-2}(i+1)x^{i+2}=x^2\sum_{i=0}^{n-2}(i+1)x^i
\end{align*}
Using well-known summation formulae, we find
\[
S_{C_n,e}(1)=\sum_{i=0}^{n-2}(i+1)=\binom{n}{2}
\]
and
\begin{align*}
S'_{C_n,e}(1)&=2\sum_{i=0}^{n-2}(i+1)+\sum_{i=0}^{n-2}i(i+1)\\
&=\sum_{i=0}^{n-2}(i+1)+\sum_{i=0}^{n-2}(i+1)^2\\
&=\binom{n}{2}+\frac{n(n-1)(2n-1)}{6}\\
&=\frac{n(n-1)(2n+2)}{6}.\qedhere
\end{align*}
\end{proof}

\section{Decreasing the mean subtree order}\label{CounterexampleSection}

This section is devoted to the proof of the following theorem.

\begin{theorem}\label{MainDecrease}
Adding an edge between two distinct, nonadjacent vertices of a connected graph can decrease the density by an amount arbitrarily close to $1/3$.
\end{theorem}

Let $\{s_n\}_{n\geq k}$ be a sequence of nonnegative integers satisfying:
\begin{enumerate}
    \item $2s_n\leq n-3$ for all $n\geq k$;
    \item $s_n=o(n)$, i.e., $\displaystyle\lim_{n\rightarrow\infty}\frac{s_n}{n}=0$; and
    \item $2^{s_n}\geq n^2$ for all $n\geq k$.
\end{enumerate}
Many such sequences exist.  Take, for example, the sequence $\{\lceil 2 \log_2(n)\rceil\}_{n\geq 32}$.
For all $n\geq k$, let $T_{n}$ be the tree obtained from a path of order $n-2s_n$ (which is at least $3$ by condition (a)) by joining $s_n$ leaves to both of the endvertices $u$ and $v$ of the path (see Figure~\ref{TnDrawing}).  Let $G_n$ be the graph obtained from $T_n$ by adding a new edge $e$ between the vertices $u$ and $v$.  We prove that
\[
\lim_{n\rightarrow\infty}\den(T_n)-\den(G_n)=\tfrac{1}{3},
\]
from which Theorem~\ref{MainDecrease} follows immediately.

\begin{figure}[htb]
\centering
\begin{tikzpicture}
\path (-3.2,-1.1) rectangle (7.2,1.1);
\vertex (0) at (0,0) {};
\vertex (1) at (1,0) {};
\vertex (3) at (3,0) {};
\vertex (4) at (4,0) {};
\vertex (a) at (0,1) {};
\vertex (b) at (-0.7,0.7) {};
\vertex (c) at (-0.7,-0.7) {};
\vertex (d) at (0,-1) {};
\vertex (e) at (4,1) {};
\vertex (f) at (4.7,0.7) {};
\vertex (g) at (4.7,-0.7) {};
\vertex (h) at (4,-1) {};
\node[rotate=90] at (-0.5,0.05) {\footnotesize $\dots$};
\node[rotate=90] at (4.5,0.05) {\footnotesize $\dots$};
\node[above right] at (0,0) {\footnotesize $u$};
\node[above left] at (4,0) {\footnotesize $v$};
\path[-...] (1) edge (2.22,0);
\path (2.22,0) edge (3);
\path
(0) edge (1)
(3) edge (4)
(a) edge (0)
(b) edge (0)
(c) edge (0)
(d) edge (0)
(e) edge (4)
(f) edge (4)
(g) edge (4)
(h) edge (4);
\draw [decorate,decoration={brace,amplitude=4pt,mirror},xshift=-3pt,yshift=0pt]
(-0.8,1.1) -- (-0.8,-1.1) node [left,black,midway,xshift=-3pt]
{\footnotesize $s_n$ leaves};
\draw [decorate,decoration={brace,amplitude=4pt},xshift=3pt,yshift=0pt]
(4.8,1.1) -- (4.8,-1.1) node [right,black,midway,xshift=3pt]
{\footnotesize $s_n$ leaves};
\draw [decorate,decoration={brace,amplitude=4pt,mirror},yshift=-3pt]
(0.1,0) -- (3.9,0) node [below,black,midway,yshift=-3pt]
{\footnotesize path of order $n-2s_n$};
\end{tikzpicture}
\caption{The tree $T_n$.}
\label{TnDrawing}
\end{figure}
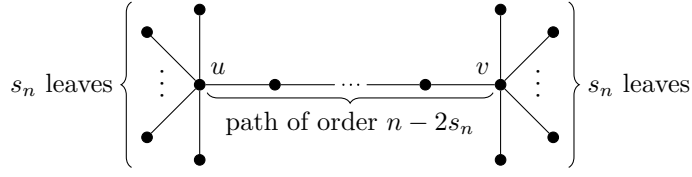

First of all, we can write
\begin{align}\label{ePartition}
S_{G_n}(x)=S_{G_n,e}(x)+S_{T_n}(x),
\end{align}
since the subtrees of $G_n$ can be partitioned into those that contain $e$ and those that do not.  This means that $\mu(G_n)$ is a convex combination (or ``weighted average'') of the means $\mu(G_n,e)$ and $\mu(T_n)$ (see~\cite[Lemma 3.8]{Jamison1983}).  To be precise, we have
\begin{align}\label{Weighted}
    \mu(G_n)=\frac{S_{G_n,e}(1)}{S_{G_n}(1)}\mu(G_n,e)+\frac{S_{T_n}(1)}{S_{G_n}(1)}\mu(T_n).
\end{align}
Dividing through by $n$, we obtain
\begin{align}\label{WeightedDensities}
    \den(G_n)=\frac{S_{G_n,e}(1)}{S_{G_n}(1)}\den(G_n,e)+\frac{S_{T_n}(1)}{S_{G_n}(1)}\den(T_n).
\end{align}
We will show that 
\[
\lim_{n\rightarrow\infty}\den(G_n)=\lim_{n\rightarrow\infty}\den(G_n,e)=\tfrac{2}{3}.
\]
On the other hand, from the proof of~\cite[Corollary 4.3]{MolOellermann2019}, we have
\begin{align}\label{MTn}
\mu(T_n)>n-s_n-1.
\end{align}
(Although $s_n$ was set equal to $\lceil 2\log_2(n)\rceil$ in~\cite{MolOellermann2019}, only the fact that $2^{s_n}\geq n^2$ is necessary for the proof.)
Since $s_n=o(n),$ it follows immediately that $\displaystyle\lim_{n\rightarrow\infty}\den(T_n)=1$.

We first compute the local mean $\mu(G_n,e)$ and show that $\displaystyle\lim_{n\rightarrow\infty}\den(G_n,e)=\tfrac{2}{3}$.  By a straightforward counting argument, we have
\begin{align}\label{LocalPoly}
S_{G_n,e}(x)=(1+x)^{2s_n}S_{C_{n-2s_n},e}(x),
\end{align}
where $C_{n-2s_n}$ is the cycle in $G_n$ induced by the $n-2s_n$ vertices of the $uv$-path in $T_n$.  By Lemma~\ref{logsum}, and then Lemma~\ref{LocalCycleLemma} and a straightforward computation, we have
\begin{align*}
\mu(G_n,e)&=\mu((1+x)^{2s_n})+\mu(C_{n-2s_n,e})\\
&=s_n+\frac{2n-4s_n+2}{3}\\
&=\frac{2n-s_n+2}{3}.
\end{align*}
It follows immediately that
\[
\lim_{n\rightarrow\infty}\den(G_n,e)=\tfrac{2}{3}.
\]

We now show that $\displaystyle\lim_{n\rightarrow \infty}\den(G_n)=\lim_{n\rightarrow\infty}\den(G_n,e).$
We begin by demonstrating that $\displaystyle\lim_{n\rightarrow \infty}\frac{S_{T_n}(1)}{S_{G_n}(1)}=0$.  Since $\displaystyle \frac{S_{T_n}(1)}{S_{G_n}(1)}$ is clearly positive, it suffices to show that $\displaystyle\lim_{n\rightarrow\infty}\frac{S_{T_n}(1)}{S_{G_n}(1)}\leq 0$.
From (\ref{ePartition}), we have $S_{G_n}(1)=S_{G_n,e}(1)+S_{T_n}(1)$.  By  Lemma~\ref{LocalCycleLemma}, evaluating (\ref{LocalPoly}) at $x=1$  gives
\[
S_{G_n,e}(1)=\binom{n-2s_n}{2}2^{2s_n}.
\]
We also use expression (9) from~\cite{MolOellermann2019}, namely
\begin{align}\label{Tn1}
S_{T_n}(1)=2s_n+\binom{n-2s_n-1}{2}+2(n-2s_n-1)2^{s_n}+2^{2s_n}.
\end{align}  
Putting all of this together, we have
\begin{align*}
    \lim_{n\rightarrow \infty}\frac{S_{T_n}(1)}{S_{G_n}(1)}&=\lim_{n\rightarrow \infty}\frac{S_{T_n}(1)}{S_{G_n,e}(1)+S_{T_n}(1)}\\
    &\leq \lim_{n\rightarrow\infty}\frac{S_{T_n}(1)}{S_{G_n,e}(1)}\\
    &=\lim_{n\rightarrow\infty}\frac{2s_n+\binom{n-2s_n-1}{2}+2(n-2s_n-1)2^{s_n}+2^{2s_n}}{\binom{n-2s_n}{2}2^{2s_n}}.
\end{align*}
Applying rough upper bounds in the numerator, and then using the facts that $\lim_{n\rightarrow\infty}s_n/n=0$ and $2^{s_n}\geq n^2$, we find
\begin{align}
    \lim_{n\rightarrow\infty}\frac{S_{T_n}(1)}{S_{G_n}(1)}&\leq \lim_{n\rightarrow\infty}\frac{n+n^2+2n\cdot 2^{s_n}+2^{2s_n}}{\binom{n-2s_n}{2}2^{2s_n}}=0.
\end{align}
It follows immediately that
\[
\lim_{n\rightarrow \infty}\frac{S_{G_n,e}(1)}{S_{G_n}(1)}=\lim_{n\rightarrow \infty}\frac{S_{G_n}(1)-S_{T_n}(1)}{S_{G_n}(1)}=1.
\]
Thus, from (\ref{WeightedDensities}), we obtain
\begin{align*}
    \lim_{n\rightarrow \infty}\den(G_n)&=\lim_{n\rightarrow \infty}\frac{S_{G_n,e}(1)}{S_{G_n}(1)}\cdot\lim_{n\rightarrow\infty}\den(G_n,e)+\lim_{n\rightarrow \infty}\frac{S_{T_n}(1)}{S_{G_n}(1)}\cdot\lim_{n\rightarrow\infty}\den(T_n)\\
    &=\lim_{n\rightarrow\infty}\den(G_n,e)\\
    &=\tfrac{2}{3}.
\end{align*}
This completes the proof that 
\[
\lim_{n\rightarrow\infty}\left[\den(T_n)-\den(G_n)\right]=\tfrac{1}{3},
\]
and Theorem~\ref{MainDecrease} follows immediately.

\section{Increasing the mean subtree order}\label{TreeSection}

For an arbitrary connected graph $G$, we have shown that it is not necessarily true that $\mu(H)>\mu(G)$ for every graph $H$ obtained from $G$ by joining two distinct, nonadjacent vertices.  However, if $G$ is not complete, then we suspect that there exists some graph $H$, obtained from $G$ by joining two distinct, nonadjacent vertices, such that $\mu(H)>\mu(G)$.  Here, we prove this statement in the special case that $G$ is a tree.

\begin{theorem}
For every tree $T$ of order $n\geq 3$, there is a graph $H$, obtained from $T$ by joining two distinct, nonadjacent vertices, such that $\mu(H)>\mu(T).$
\end{theorem}

\begin{proof}
Let $u$ be a vertex of $T$ such that at least two components of $T-u$,  say $P$ and $Q$, have the property that $T[V(P)\cup \{u\}]$ and $T[V(Q)\cup\{u\}]$ are paths (see Figure~\ref{duck}).  Such a vertex $u$ is guaranteed to exist.  Suppose that $P$ has order $p$, and $Q$ has order $q$.  Let $v$ be the vertex of $P$ adjacent to $u$ in $T$, and let $w$ be the vertex of $Q$ adjacent to $u$ in $T$.  Let $R=T-(V(P)\cup V(Q))$.

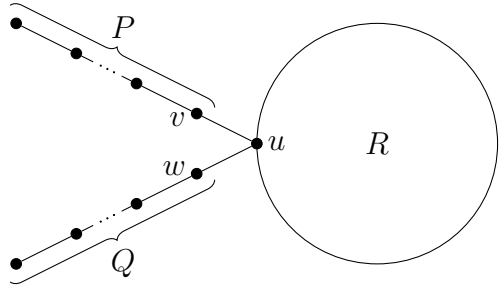
\begin{figure}[htb]
\begin{center}
\begin{tikzpicture}[scale=0.8]
\vertex (u) at (0,0) {};
\draw (0,0) node[right]{$u$};
\vertex (v) at (-1,0.5) {};
\draw (-1,0.5) node[left, yshift=-2pt]{$v$};
\vertex (v1) at (-2,1) {};
\vertex (v2) at (-3,1.5) {};
\vertex (v3) at (-4,2) {};
\vertex (w) at (-1,-0.5) {};
\draw (-1,-0.5) node[left, yshift=2pt]{$w$};
\vertex (w1) at (-2,-1) {};
\vertex (w2) at (-3,-1.5) {};
\vertex (w3) at (-4,-2) {};
\draw (2,0) circle (2) node{$R$};
\path[-...] 
(v1) edge (-2.72,1.36)
(w1) edge (-2.72,-1.36);
\path
(-2.72,1.36) edge (v2)
(-2.72,-1.36) edge (w2)
(u) edge (v)
(v) edge (v1)
(v2) edge (v3)
(u) edge (w)
(w) edge (w1)
(w2) edge (w3);
\draw [decorate,decoration={brace,amplitude=4pt}]
(-4.1,2.3) -- (-0.7,0.6) node [above,black,midway,xshift=4pt,yshift=3pt]
{$P$};
\draw [decorate,decoration={brace,amplitude=4pt,mirror}]
(-4.1,-2.3) -- (-0.7,-0.6) node [below,black,midway,xshift=4pt,yshift=-3pt]
{$Q$};
\end{tikzpicture}
\end{center}
\caption{The tree $T$.}\label{duck}
\end{figure}

Let $H$ be the graph obtained from $T$ by joining vertices $v$ and $w$; call this new edge $e$.  We claim that $\mu(H)>\mu(T)$.  Since 
\[
S_{H}(x)=S_{H,e}(x)+S_T(x),
\]
we see that $\mu(H)$ is a weighted average of $\mu(H,e)$ and $\mu(T)$, so it suffices to show that $\mu(H,e)>\mu(T).$
By~\cite[Theorem 3.9]{Jamison1983}, for all $u\in V(T)$, we have
\[
\mu(T,u)>\mu(T),
\]
so it suffices to show that $\mu(H,e)\geq \mu(T,u).$  For any nonnegative number $k$, let $f_k(x)=\sum_{i=0}^k x^i$.  By straightforward counting arguments, we have
\begin{align}
    S_{H,e}(x)
    &=x^2f_{p-1}(x)f_{q-1}(x)(1+2S_{R,u}(x))\label{f1}
\end{align}
and
\begin{align}
    S_{T,u}(x)
    &=f_p(x)f_q(x)S_{R,u}(x).\label{f2}
\end{align}
Now for any $k\geq 0$, by a straightforward computation, we have $\mu(f_k(x))=k/2$.  Applying Lemma~\ref{logsum} to (\ref{f1}) and (\ref{f2}), we find
\begin{align*}
\mu(H,e)&=\mu(x^2)+\mu(f_{p-1}(x))+\mu(f_{q-1}(x))+\mu(1+2S_{R,u}(x))\\
&=2+\tfrac{p-1}{2}+\tfrac{q-1}{2}+\frac{2S'_{R,u}(1)}{1+2S_{R,u}(1)}\\
&=\tfrac{p}{2}+\tfrac{q}{2}+1+\frac{2S_{R,u}(1)}{1+2S_{R,u}(1)}\mu(R,u).
\end{align*}
and
\[
\mu(T,u)=\mu(f_p(x))+\mu(f_q(x))+\mu(R,u)=\tfrac{p}{2}+\tfrac{q}{2}+\mu(R,u),
\]
respectively.  So it suffices to show that
\[
1+\frac{2S_{R,u}(1)}{1+2S_{R,u}(1)}\mu(R,u)>\mu(R,u),
\]
or equivalently,
\[
1+2S_{R,u}(1)>\mu(R,u).
\]
Since $S_{R,u}(1)$ is at least the order of $R$, while $\mu(R,u)$ is at most the order of $R$, the result follows immediately. 
\end{proof}

Note that in the statement of Conjecture~\ref{AddConjecture}, we stipulate that the graph $H$ must be obtained from $G$ by joining two \emph{nonadjacent} vertices.  If this condition is dropped, then the statement becomes much easier to prove.  We show that for every multigraph $G$, there is a multigraph $H$, obtained from $G$ by adding an edge between a pair of distinct (but possibly adjacent) vertices of $G$, such that $\mu(H)>\mu(G)$.


We first require a lemma that follows almost directly from Theorem~3.2 of \cite{Andriantiana2019}. We introduce some notation used in the proof of this lemma. Let $\mathcal{T}_{G}^{*}$ be the set of all subtrees of $G$ of order at least $2$, and define $S_{G}^{*}(x)$ by 
\[
S_{G}^{*}(x)=\sum_{T\in\mathcal{T}_{G}^*}x^{|V(T)|}.
\]
 The mean subtree order of all trees of order at least $2$ of $G$ is then defined analogously as 
 \[
 \mu^{*}(G)=\frac{S_{G}^{*\prime}(1)}{S_{G}^{*}(1)}.
 \]
 Clearly, we have $\mu^*(G)>\mu(G)$ for every multigraph $G$.

\begin{lemma}\label{DecMeanEdgeDeletionLemma}
If $G$ is a multigraph  with $E(G)\neq\emptyset$, then there exists an edge $e\in E(G)$ such that $\mu(G,e)>\mu(G)>\mu(G-e)$.
\end{lemma}
\begin{proof}
Let $G$ be a multigraph of order $n$ with at least one edge.  If $G$ contains no subtree of order $3$, then for every edge $e\in E(G)$, we have $\mu(G,e)=2>\mu(G)$.  So we may assume that $G$ contains a subtree of order $3$.

Let $\mathcal{B}$ be the set of all nonempty subsets of $E(G)$ (multiple edges distinguished) that induce a subtree of $G$.  Since $G$ contains subtrees of orders $2$ and $3$, the cardinalities of the elements of $\mathcal{B}$ are not all the same.  Therefore, by \cite[Theorem 3.2]{Andriantiana2019}, it follows that there exists an edge $e\in E(G)$ such that $\mu^{*}(G)>\mu^{*}(G-e)$.  Since $S_G^{*}(x)=S_{G,e}(x)+S_{G-e}^{*}(x)$, we have that $\mu^{*}(G)$ is a weighted average of $\mu(G,e)$ and $\mu^{*}(G-e)$, so $\mu(G,e)>\mu^{*}(G)>\mu(G)$.  The fact that $\mu(G)>\mu(G-e)$ now follows from the fact that $\mu(G)$ is a weighted average of $\mu(G,e)$ and $\mu(G-e)$.
\end{proof}

One might hope that Lemma~\ref{DecMeanEdgeDeletionLemma} could be used to show that if $G$ has minimum mean subtree order among all connected graphs of a given order, then $G$ is a tree. We remark, however, that this fact does not follow immediately from Lemma~\ref{DecMeanEdgeDeletionLemma}.  While Lemma~\ref{DecMeanEdgeDeletionLemma} guarantees that every graph $G$ contains an edge $e$ whose deletion decreases the mean subtree order, it does not guarantee that the edge $e$ lies on a cycle of $G$.
We can, however, use Lemma~\ref{DecMeanEdgeDeletionLemma} to demonstrate the following result.

\begin{proposition}
Let $G$ be a multigraph of order at least $2$.  Then there is a multigraph $H$, obtained from $G$ by adding a new edge between a pair of distinct vertices of $G$, such that $\mu(H)>\mu(G)$.
\end{proposition}
\begin{proof}
If $G$ has no edges, then for every pair of distinct vertices $u,v\in V(G)$, the graph $H$ obtained from $G$ by joining $u$ and $v$ satisfies $\mu(H)>\mu(G)$.  So we may assume that $E(G)\neq \emptyset$.  By Lemma~\ref{DecMeanEdgeDeletionLemma}, there is an edge $e\in E(G)$ such that $\mu(G,e)>\mu(G)$.  Let $H$ be the graph obtained from $G$ by adding a new edge $f$ between the endvertices of $e$.  Note that no subtree of the multigraph $H$ contains both $e$ and $f$, since they induce a cycle. Therefore, we have $S_{H,f}(x)=S_{G,e}(x)$.  So we can write
\[
S_H(x)=S_{H,f}(x)+S_{G}(x)=S_{G,e}(x)+S_G(x).
\]
It follows that $\mu(H)$ is a weighted average of $\mu(G,e)$ and $\mu(G)$. Since $\mu(G,e)>\mu(G)$, we must have $\mu(H)>\mu(G)$.
\end{proof}
\section{Conclusion}

While we have shown that adding an edge to a graph can drastically decrease its mean subtree order, we do not know whether our examples are extremal.

\begin{problem}
Suppose that a graph $H$ is obtained from a connected graph $G$ by adding an edge between two nonadjacent vertices of $G$.  Determine sharp bounds on $\den(H)-\den(G).$
\end{problem}

\noindent
We have shown that the difference $\den(H)-\den(G)$ can be arbitrarily close to $-1/3$.  On the other hand, we remark that $\den(C_n)-\den(P_n)$ approaches $1/6$ asymptotically (see~\cite[Corollary 3.2]{ChinGordonMacpheeVincent2018}).  We suspect that these are in fact the extremal values for the difference.

Another open problem is to decide whether or not Conjecture~\ref{AddConjecture} holds in general.  We have shown that it holds in the special case that $G$ is a tree, and that it holds in general if we drop the requirement that the two vertices being joined are nonadjacent.  If true, Conjecture~\ref{AddConjecture} would imply that among all connected graphs of order $n$, the complete graph $K_n$ has the maximum mean subtree order, while the path $P_n$ has the minimum mean subtree order.  (Jamison~\cite{Jamison1983} demonstrated that $P_n$ has the minimum mean subtree order among all trees of order $n$.)  This remains a significant open problem on its own.


\end{document}